\documentclass[11pt]{amsart}

\def\zz{{\bf Z}}
\def\qq{{\bf Q}}

\def\calc{\mathcal{C}}
\def\calg{\mathcal{G}}

\newtheorem{theorem}{Theorem}[section]
\newtheorem{lemma}[theorem]{Lemma}
\newtheorem{corollary}[theorem]{Corollary}

\theoremstyle{definition}

\theoremstyle{remark}

\numberwithin{equation}{section}

\date{today}

\begin{document}

\title{Stabilizing  Four--Torsion in Classical Knot Concordance}

\author{Charles Livingston}

\address{Department of Mathematics, Indiana University, Bloomington, IN 47405}

\email{livingst@indiana.edu}

\author{Swatee Naik}

\address{Department of Mathematics, University of Nevada, Reno, NV 89557}

\email{naik@unr.edu}
\thanks{The second author was supported in part by NSF Grant \#DMS-0709625.}
\keywords{\texttt{Knot, Concordance}}

\subjclass{57M25, 57N70}

\date{\today}

\begin{abstract} Let $M_K$ be the 2--fold branched cover of a knot $K$ in $S^3$.  If
$H_1(M_K) = \zz_3 \oplus \zz_{3^{2i}} \oplus G$ where 3 does not divide the order of $G$ then
$K$ is not of order 4 in the concordance group.  This obstruction detects infinite new families of  knots that represent elements of order 4 in the algebraic
concordance group that are not of order 4 in concordance.

\end{abstract}

\maketitle

\section{Introduction}

Levine \cite{le1, le2} defined a  homomorphism $\phi$ from  the concordance group $\calc$ of knots in $S^3$, onto an algebraically defined group $\calg$, and further 
proved that 
$\calg  \cong
\oplus\   \zz ^\infty \oplus   {\zz_2}^\infty \oplus {\zz_4 }^\infty$.  It is a long standing conjecture that $\calc$ contains no torsion of order
other than two;  see for instance \cite{fm, k}. This paper continues our investigation of the possibility of
elements of order four in $\calc$.

For a knot $K \subset S^3$, let 
$M_K$ denote the 2--fold branched cover of $S^3$ branched over $K$, and for a prime $p$, let
$H_1(M_K)_p$ denote  the $p$--primary subgroup of $H_1(M_K)$; homology is with integer
coefficients throughout this paper. Our earlier work on 4--torsion,~\cite{ln1, ln2},  demonstrated the following.   

\begin{theorem}\label{oldthm} If 
  $H_1(M_K)_p \cong \zz_{p^k}$  for some prime $p \equiv 3 \mod 4$ with $k$ odd, then $K$ is of 
infinite order in $\calc$.
\end{theorem}
This criterion is effective in ruling out the possibility of being order four for most low-crossing knots that represent four torsion in $\calg$.

Since we wrote~\cite{ln2}, several papers have appeared that apply new methods in smooth concordance theory (in particular Heegaard-Floer theory) to the study of 4--torsion.  This work includes~\cite{grs, jn, lisca}.  Given the continued interest in the structure of the concordance group, we here investigate the extension of our earlier work to the case in which $H_1(M_K)_p$ is not cyclic.  Working with primes greater than three greatly complicates the algebra; our main result is restricted to the case of $p = 3$.

\begin{theorem} \label{main}  If $H_1(M_K)_3 \cong \zz_3 \oplus {\zz_{3^{2i}}}$  then
$K$ is not of order 4 in $\calc$.
\end{theorem}

We will also present  applications of this result, describing new infinite   families of knots that are of algebraic order four but do not represent 4--torsion in $\calc$.  A simple, easily stated application is the following, where  the  Alexander polynomial of a knot $K$ is denoted
$\Delta_{K}(t)$:

\begin{corollary} If $\Delta_K(t)$ is quadratic and $\Delta_K(-1) = 27m$ where 3 does not divide $m$, then $K$ is of order 4 in $\calg$  but not in $\calc$.  
\end{corollary}

While the  simplest application of our main result  is to prove that particular  knots that are
of algebraic order four are not of order four in $\calc$, we are more interested in the fact that this obstruction applies to   entire $S$--equivalences classes of knots, and thus the calculation of the obstruction is purely algebraic, based on simple classical algorithms from knot theory.

   Of further interest is that the result applies in the topological, locally flat category.  The techniques we use are based on Casson-Gordon theory, which initially applied only in the smooth category (see, for example, \cite{cg1, cg2}), but  by \cite{fq} the techniques extend to the
topological locally flat category.   With regards to examples  taken from low-crossing prime knots, all   algebraic order four knots that have been shown to be of order greater than 4  smoothly can be shown to have order greater than four topologically.

  Basic results in knot theory can be
found in \cite{ro} or \cite{go}.  Tables of low crossing knots
and their algebraic  and concordance orders can be found in \cite{cl}.

\section{\bf Casson-Gordon invariants and linking forms}

  Let
$\chi$ denote a homomorphism from
$H_1(M_K)$ to
$\zz_{p^k}$, for some prime $p$.  The Casson-Gordon invariant 
$\sigma(K,\chi)$ is a rational invariant of the pair
$(K,\chi)$.  (See \cite{cg1, cg2}. In the original paper, [CG1], this
invariant is denoted $\sigma_1 \tau(K,\chi)$, and $\sigma$ is used for
a closely related invariant.)

On the rational homology sphere  $M_K$  there is a nonsingular
symmetric linking form, $\beta :H_1(M_K) \rightarrow \qq / \zz$. For a subgroup $M
\subset H_1(M_K)$ we let $M^\perp = \{ x \in H_1(M_K)\ |\ \beta(x,m) = 0\ \forall\ m
\in M\}$.
  The main result
in [CG1]
concerning Casson-Gordon invariants and slice knots that we will be
using is the following:

\begin{theorem} \label{CG}
If $K$ is slice there is a subgroup 
$M\subset H_1(M_K) $ with $M = M^\perp$ and
$\sigma(K,\chi) = 0$ for all prime power order
$\chi$ vanishing on $M$.\end{theorem}

\noindent

A subgroup $M \subset \ H_1(M_K)$ satisfying $M = M^\perp$ is called a {\it
metabolizer}.  It is useful to recall the following result.

\begin{lemma}\label{H/M=M} For a metabolizer    $M \subset H_1(M_K)$,
$H_1(M_K)/M \cong M$ and in particular $|M|^2 = |H_1(M_K)|$.

\end{lemma}

\begin{proof} This follows quickly from the following exact sequence  $$0 \to M^\perp
\to H_1(M_K) \to
\hom(M,\qq / \zz) \to 0,$$ the fact that
$M^\perp = M$, and the observation that since $M$ is a  finite abelian group,
$\hom(M,\qq /
\zz)
\cong M$.

\end{proof}
 We will need Gilmer's additivity theorem  \cite{gl}, a vanishing result
proved by Litherland \cite[Corollary
B2]{lit}, and a simple fact that follows immediately from the definition of the
Casson--Gordon invariant.

\begin{theorem}  \label{Gi}
If $\chi_1$ and $\chi_2$ are defined on $M_{K_1}$
and   $M_{K_2}$, respectively, then we have
$\sigma( K_1\ \#\ K_2, \chi_1\ \oplus \
\chi_2) =
\sigma(K_1,\chi_1) + \sigma(K_2,\chi_2)$. \end{theorem}

\begin{theorem} If $\chi$ is the trivial character, then $\sigma(K,\chi) = 0$.
\end{theorem}
\begin{theorem} For every character  $\chi$, 
$\sigma(K,\chi) = \sigma(K,-\chi)$.
\end{theorem}

We will also need to use the relationship between  the Casson--Gordon
invariant of a knot   and the linking form on its 2--fold branched cover, as developed
in \cite{ln1, ln2}.

\begin{theorem} \label{cglink} If $\chi \colon H_1(M_K) \to \zz_{p^r}$
is a character obtained by linking with the element $x \in H $,
then $\sigma(K,\chi) \equiv \beta(x,x)\ \mbox{\rm modulo}\  \zz$.   
\end{theorem}

\noindent
This will be used later to conclude that certain Casson--Gordon invariants are nonzero.

\vskip.2in 
\noindent {\bf NOTATION:}\\
 In the rest of this paper all knots will satisfy 
$H_1(M_K)_3
\cong
\zz_3
\oplus \zz_{3^{2i}}$. All characters $\chi$ will   take values in $\zz_{3^{2i}} \subset
\qq/\zz$, and such $\chi$ factor through characters defined on $\zz_3
\oplus \zz_{3^{2i}}$. Any such character is given by linking with an element of the
$H_1(M_K)_3$, say $(x,y) \in \zz_3
\oplus \zz_{3^{2i}}$.  To simplify notation we will write $\sigma(K, \chi) $ as
$\sigma_{x,y}.$

\section{Proof of Theorem 1.2}

Throughout this section we will assume that  $4K$ is slice.  We will consider all
possible metabolizers to the linking form on $(\zz_3 \oplus  \zz_{3^{2i}})^4$ and show that
each leads to a contradiction to Theorem \ref{CG}.

\begin{lemma} There is a generating set $ \{ v , w\} $ for $\zz_3 \oplus  \zz_{3^{2i}}$
  such that $v$ is of order 3, $w$ is of order $3^{2i}$, and the linking form
satisfies: $\beta(v,v) =\pm 1/3$,  $\beta(w,w) =\pm 1/3^{2i}$, and $\beta(v,w) =0$.
\end{lemma}

\begin{proof}   Let $a$ generate the $\zz_3$ summand and let $b$
generate the
$\zz_{3^{2i}}$ summand. Since there is a character to $\qq / \zz$ taking value
$1/3^{2i}$ on $b$, by the nonsingularity of the linking form there is an element $x$
satisfying
$\beta(x,b) = 1/3^{2i}$ .  Write  $x = ra + sb$.  Since $\beta(a,b) $ is a multiple
of $1/3$ ($a$ is of order 3), $s \beta(b,b)$ must of the form  $t/3^{2i}$ with $t$
not divisible by 3.  Hence there is an integer $u$ such that $u\beta(b,b) =
 1/3^{2i}$.  
Let $v =a -  3^{2i}\beta(a,b)   u b$. It is easily checked that  $v$ is of order 3
and $\beta(v,b) = 0$.  

By the nonsingularity of the linking form, $\beta(v,v)= \pm 1/3$.  As observed above, 
$\beta(b,b)
= t/3^{2i}$ for some  
$t \in \zz_{3^{2i}},\ t \, \not\equiv \, 0$ mod $3$.    Let $s$ be the inverse to $t$ in 
$\zz_{3^{2i}}$. Then $\pm s = q^2$ for some $q \in \zz_{3^{2i}}$.  
(The square of an element is $0$ mod $3$ if and only if the element itself
is such. 
In $\zz_{3^{2i}}$ there are a total of  $3^{2i-1}$ elements which are $0$ mod
$3$. It follows that there are $3^{2i}-3^{2i-1}$
elements which are $\pm 1$ mod 3, half of which 
are additive inverses of the other half, and there are  
$\frac{3^{2i}-3^{2i-1}}{2}$
 distinct squares which are not $0$ mod $3.$)
  Let $w = qb$.

\end{proof}

From now on we will fix the generating set to be as given in the previous lemma.

In order to apply Theorem \ref{CG} to the knot $4K$, we let $H = H_1(M_{4K})_3
\cong 
 (\zz_3 \oplus \zz_{3^{2i}})^4 \cong ({\zz_3})^4 \oplus (\zz_{3^{2i}})^4$. We will
let $M$ denote a metabolizer in $H$.  To set up notation, we will represent an element
in
$({\zz_3})^4
\oplus ({\zz_{3^{2i}}})^4$ by an ordered 8--tuple  and a collection of $n$
elements in 
$({\zz_3})^4 \oplus ({\zz_{3^{2i}}})^4$ by an $n \times 8$ matrix, the rows of which
represent the individual elements. Each element will be written as  $$u_i = v_i
\oplus w_i \in ( {\bf Z}_{3})^4 \oplus ({\bf Z}_{{3^{2i}}} ) ^4 , 1\le i \le 4.$$

\begin{lemma}\label{rkM} Let $M$ be a metabolizer for $H$.  
Then $M$ cannot be generated by less than four elements. 

\end{lemma}

\begin{proof} Tensor $H$ and $M$ with $\zz_3$. We have $H \otimes \zz_3 \cong
({\zz_3})^8.$ 
If $M$ is generated by $k$ elements, then $M  \otimes \zz_3 \cong
({\zz_3})^k.$ If $k \le 3$, then ${\rm rk}((H \otimes \zz_3 ) / (M \otimes \zz_3)) \ge 5.$
As
rk$((H/M) \otimes \zz_3)\, \ge\, {\rm rk}((H \otimes \zz_3 ) / (M \otimes \zz_3)),$ we have
a contradiction to the fact that $H/M\,  \cong \, M.$ 
\end{proof}

We will call the minimum number of elements required to generate $M$, the rank of $M$.
The proof of Theorem 1.2 is simplest in the case that the rank is greater than 4.

\begin{theorem} \label{rkge4}If {\rm rank}$(M) = k$, $k > 4$, then $K$ is not of order
4 in concordance.
\end{theorem}

\begin{proof}
Consider a minimal generating set $\{(v_i,w_i)\}_{i = 1 \ldots k}$. These form the 
rows of a
$k
\times  8$ matrix which we denote $( V | W)$, where $V$ and $W$ are each $k \times 4$.
We will now perform row operations to simplify the generating set.  It will be
convenient to interchange columns in these matrices as well,  but notice that if  
two columns of $W$ are interchanged, the same columns of $V$ will be interchanged,
since these columns correspond  to  the homology of the cover of a given component of
$4K$.

By performing row operations and column interchanges, $W$ can be put in upper
triangular form.  Hence, the fifth row of $W$ is the trivial vector, $(0,0,0,0) \in
(\zz_{3^{2i}})^4$. After
further column swaps, the fifth row of $V$ can be put in the form $(\pm 1, \pm 1, \pm
1, 0)$, as  these are the only nontrivial elements in $({\zz_3})^4$ with trivial
self--linking.  

It follows that $3 \sigma_{1,0} = 0$, and hence $\sigma_{1,0} = 0$. However, by
Theorem \ref{cglink},   $\sigma_{1,0} \equiv 1/ 3 \ {\rm mod\ } \zz$, giving a contradiction. 
\end{proof}

The rest of this section is devoted to the case that rank$(M) = 4$.

\begin{lemma} \label{basicform} Let {\rm rank}$(M) = 4$. Then $M$ has 
a generating set $\{\, u_j=v_j \oplus w_j \in (\zz_3)^4 \oplus (\zz_{3^{2i}})^4
\, \vert \, j=1,2,3,4\, \}$  
such that the corresponding matrix $(V|W)$ is of the form given below.  The
$v_{i,j}$ are elements in $\zz_3$ and  the $w_{i,j}$
are elements in $\zz_3^{2i}$.
$$
 \left( \begin {array}{rrrr} v_{1,1}  &v_{1,2 } &v_{1,3} & v_{1,4} \\\noalign{\medskip}
 v_{2,1}  &v_{2,2 } &v_{2,3} & v_{2,4} \\\noalign{\medskip}
 v_{3,1}  &v_{3,2 } &v_{3,3} & v_{3,4} \\\noalign{\medskip}
 v_{4,1}  &v_{4,2 } &v_{4,3} & v_{4,4}\end {array}\right |
 \left. \begin {array}{rrrr} 1&0 & w_{1,3 }& w_{1,4} \\\noalign{\medskip}
 0&1&w_{2,3}  &w_{2,4}  \\\noalign{\medskip}
0&0&3^{2i-1} &0  \\\noalign{\medskip}
 0&0&0&3^{2i-1} \end {array}\right )$$
\end{lemma}

\begin{proof} 
 Row operations and column swaps (provided the same column swaps are made in $V$ as in $W$) 
 can be used to make $W$ upper triangular with the
diagonal entries nondecreasing powers of 3 such that the remaining entries
in the $j$th row are annihilated by the same power of 3 as is the diagonal entry. 
Let the diagonal entries be $3^{k_j}$ with $0 \le k_1 \le k_2 \le k_3 \le k_4 \le 2i$. 
It is easily seen that the order of the element $u_j$ 
represented by row $j$ of this matrix is $3^{2i-k_j}$ 
and together the $u_j$ generate a subgroup of order $3^{(8i - \sum k_j)}$.  On the other
hand, the order of $H$ is $3^{8i+4}$ and $M$ has the square root order $3^{4i+2}$.  It follows that 
$\sum k_j= 4i-2$.

We first note that $k_4 \ne
2i$: 
If $k_4 = 2i$ then the last   row has the form
$(v_{4,1},v_{4,2} ,v_{4,3}, v_{4,4}\ |\ 0,0,0,0)$ with some of the $v_{4,j}$ nonzero. 
Since the self--linking of this element is 0, exactly 3 of the entries would be
nonzero and it would follow that $3\sigma_{1,0} = 0$, implying that  $ \sigma_{1,0} =
0$,  contradicting Theorem \ref{cglink}.

Hence, we have $0 \le k_1 \le k_2 \le k_3 \le k_4 \le 2i-1$.  

If $k_4<2i-1$, then the generator $u_4$ generates a cyclic
subgroup of order greater
than 3. As $\sum k_j= 4i-2$, $k_4$ cannot be zero. 
It follows that 
$H/\langle\, u_4\, \rangle$ has rank 8. This implies that  
$H/\langle\, u_1,u_2,u_3,u_4\, \rangle$ has rank 5 or more. However, 
by Lemma \ref{H/M=M},  the rank of $H/M$ is $4$.
Therefore, we have
$k_4=2i-1,\  0 \le k_1 \le k_2\le k_3 \le 2i-1,$  and $k_1+k_2+k_3=2i-1.$ 
As $k_3$ cannot be 0 either, a similar argument 
shows  that $k_3$ will have to 
be  $ 2i-1.$ Therefore we have   $k_1 = k_2 =0.$ 
It is easy to see that the entries above the $1$ in the second row
and the $3^{2i}$ in the last row can be made $0$.
\end{proof}

Our argument continues to  proceed by ruling out possible metabolizers under the
assumption that $4K$ is slice.

\begin{lemma}  \label{obs0}
Each of the entries $w_{i,j}$ in $(V |W)$ in the form given by Lemma
\ref{basicform} may be assumed to be $\pm 1$ mod $3$. 
The $\zz_3$ reductions
of the elements
$(0,0, w_{1,3 }, w_{1,4} )$ and 
$(0,0, w_{2,3 }, w_{2,4} )$ are linearly independent in $(\zz_3)^4$.
\end{lemma}
\begin{proof} The self--linking of the first row is  computed to be ${\alpha \over 3}
\pm  {(1+w_{1,3}^2+w_{1,4}^2) \over 3^{2i}} $ 
where $\alpha$ is
determined by the self-linking of the $v_{1,j}$. If either 
$w_{1,3}$ or $w_{1,4}$   
 were $0$ mod $3$ then it is easily shown that this sum could not
be an integer; basically, 0 is not  the sum of two nontrivial squares modulo 3.  It
follows that neither $w_{1,3}$ nor $w_{1,4}$ can be 0.  A similar argument applies
for $w_{2,3}$ and $w_{2,4}$ .

If the   elements $(0,0,w_{1,3 }, w_{1,4} )$ and 
$(0,0, w_{2,3 }, w_{2,4} )$ were dependent over $\zz_3$, then by combining the
first two rows of $(V|W)$ we would have 
$(*, *,*,*\ \vert \ \pm 1,\pm 1, 3a , 3b)$.  But such an element cannot have self--linking 0.
\end{proof}

\begin{lemma} The metabolizer $M$ contains an element of the type $(1,1,*,*\ |\  0,0,3^{2i-1}m,3^{2i-1}n),$
where $m$ and $n$ are integers. 
\end{lemma}

\begin{proof} 
Let $v_{i,j},\ w_{i,j}$  
and $u_i=v_i \oplus w_i$ be as in Lemma \ref{basicform} 
 
Suppose that $(v_{3,1}, v_{3,2})$ and
$(v_{4,1}, v_{4,2})$ are linearly dependent in $(\zz_3)^2$.  Then a nontrivial
combination of $u_3$ and $u_4$  
would yield an element $(0,0,*,*\ |\ 0,0,3^{2i-1}m, 3^{2i-1}n) \in M$.  Note that non-triviality in this case is over $\zz _3$. 
In other words, either $m$ or $n$ is nonzero mod 3. To have self--linking zero the $*$ entries would have to be 0, so that we 
have $u=(0,0,0,0\ |\ 0,0,3^{2i-1}m, 3^{2i-1}n) \in
M$.  

Now, from Lemma \ref{obs0} $(w_{1,3},w_{1,4})$ and $(w_{2,3},w_{2,4})$ are
linearly independent over $\zz _3$, so a linear combination of these yields a 
vector whose $\zz _3$ reduction is $(1,0).$ As the corresponding 
 linear combination of $u_1, u_2$ 
is an element in $M$ and therefore links the above $u$  
 trivially, we have
 $m \equiv 0$ mod $3$. Similarly $n \equiv 0$ mod $3$, giving us a contradiction.

It follows that $(v_{3,1}, v_{3,2})$ and
$(v_{4,1}, v_{4,2})$ are independent over $\zz _3$. Now,  by taking an appropriate combination of 
$u_3$ and $u_4$ 
we can find the desired element of $M$.
\end{proof}

\begin{lemma}   For $a,b \in \{0,\pm 1\},$  $M$ contains elements of the form
$(1,1,*,*\ |\ 3^{2i-1}a, 3^{2i-1}b,3^{2i-1}m,3^{2i-1}n),$
where $m,n \in \zz$ and exactly one of
the $*$ entries is nonzero.
\end{lemma}
\begin{proof} Add $3^{2i-1}a$ times the first row and
$3^{2i-1}b$ times the second row  of the matrix 
to the element given in the previous
lemma. The condition on  the first two $*$s  
comes from the fact that the self--linking of
the resulting element must be 0.
\end{proof}

\begin{proof}[{\bf COMPLETION OF PROOF, THEOREM \ref{main}}]

By Theorem \ref{cglink}, 
$ \sigma_{1,0},\ \sigma_{1,3^{2i-1}}$ and $\sigma_{1,2 \cdot 3^{2i-1}}$ are nonzero.

 From the previous lemma we
have,
 in the case $a = b = 0$,   that  either $3\sigma_{1,0} = 0$,  
$2\sigma_{1,0} +
\sigma_{1,3^{2i-1}} = 0$ or 
$2\sigma_{1,0} +
\sigma_{1,2 \cdot 3^{2i-1}} = 0$. The possibility that  $3\sigma_{1,0} = 0$ contradicts
Theorem \ref{cglink}, so either 
$2\sigma_{1,0} +
\sigma_{1,3^{2i-1}} = 0$, or  $2\sigma_{1,0} +
\sigma_{1,2 \cdot 3^{2i-1}} = 0$. 

Similarly, by letting $a =b =1$ we have either $2\sigma_{1,3^{2i-1}} +
\sigma_{1,0} = 0$ or 
$2\sigma_{1,3^{2i-1}} +
\sigma_{1,2 \cdot 3^{2i-1}} = 0$.

Finally,  letting $a =b = -1$ we have either $2\sigma_{1,2 \cdot 3^{2i-1}} +
\sigma_{1,0} = 0$ or 
$2\sigma_{1,2 \cdot 3^{2i-1}} +
\sigma_{1,3^{2i-1}} = 0$.

Considering the two relations $2\sigma_{1,0} +
\sigma_{1,3^{2i-1}} = 0$ and  $2\sigma_{1,3^{2i-1}} +
\sigma_{1,0} = 0$ together,  it follows that $3\sigma_{1,0} = 0$, contradicting
Theorem \ref{cglink}.  Similar considerations with pairs of relations rule out several possibilities.

Only two possibilities remain: the first is that  
$2\sigma_{1,0} + \sigma_{1,3^{2i-1}} = 0$, $2\sigma_{1,3^{2i-1}} + \sigma_{1,2 \cdot 3^{2i-1}} = 0$, and
$2\sigma_{1,2 \cdot 3^{2i-1}} + \sigma_{1,0} = 0$; 
the second is that $2\sigma_{1,0} + \sigma_{1,2 \cdot 3^{2i-1}} = 0$, 
$2\sigma_{1,3^{2i-1}} + \sigma_{1,0} = 0$, and $2\sigma_{1,2 \cdot 3^{2i-1}} + \sigma_{1,3^{2i-1}} = 0$.  Either
case quickly implies that $3^{2i}\sigma_{1,0} = 0$, so $\sigma_{1,0} =0$, again
contradicting \ref{cglink}.
\end{proof}

\section{Applications}\label{applicationssec}

Consider a knot with Alexander polynomial $\Delta_K(t) = kt^2 - (2k+1)t + k$, $k
\ge 0$ According to Levine \cite{le2} such a knot has finite order in the algebraic
concordance group.  It will have algebraic concordance order 4 if and only if there is some prime congruent
to 3 mod 4 which has odd exponent in $4k+1$.  According to \cite{ln1}, if $4k+1 = 3m$
with $m$ prime to 3 then $K$ is not of order 4 in concordance. 
We have the following extension.

\begin{corollary}  \label{application}
If $\Delta_K(t) = kt^2 -(2k+1)t + k$ and $4k+1 = (3^{2n+1})m$ with $n=0$ or $1$ and $m$
prime to 3 then $K$ is not of order 4 in concordance.
\end{corollary}

\begin{proof} The case $n=0$ is settled by \cite{ln1}. 
So let $n=1.$ Since the Alexander polynomial is quadratic, $H_1(M_K)$ is of rank at
most 2.  In the case that the rank is 1, then $H_1(M_K)_3  \cong  \zz_{27}$ 
and hence
the main theorem of \cite{ln2} applies to show that $K$ is not of order 4.   In the
case that the rank of $H_1(M_K)_3$ is 2, then $H_1(M_K)_3 \cong \zz_3\oplus \zz_9$ and
Theorem \ref{main} applies. 
\end{proof}

\noindent {\bf Doubled Knots}  According to \cite{cg1,cg2} the $k$--twisted double of
the unknot, $D_k$, is algebraically slice if and only if $4k+1 = l^2$ for some integer
$l$.  We are thus interested in the case that $4k+1 = 9m^2$ with $m$ prime to 3.  For
this to occur, $m$ must be odd: $m = 2n+1$.  Solving gives $k = 9(n^2 +n) +2$. 
Furthermore, $m$ will be prime to 3 if $ n \ne 1$ mod 3. 

A similar calculation shows that $D_k$ satisfies $H_1(D_K) \cong \zz_3 \oplus \zz_m$ 
with $m$ prime to 3 if $k = 3n +2$ with $n \ne 0$ mod 3.   Hence, we have the
corollary:

\begin{corollary} For all positive $r \ne 0$ mod 3 and positive $s \ne 1$ mod 3, the
knot
$D_{3r+2}
\# D_{9(s^2+s)+2}$ is of algebraic order 4 but is not of order 4 in concordance.
\end{corollary}

Finally, results of this paper apply to S-equivalence classes of knots. To show that the 
algebraic concordance class of a knot $K$ 
 cannot be realized by a knot of concordance order 4,  we need to consider knots 
with the same Seifert form as $K \# J$, where $J$ is algebraically slice.  The present
paper marks the first progress in that direction, by showing that if
 $H_1(M_K) \cong \zz_3$ and $J$ is an algebraically
slice knot with $H_1(M_J)_3$ cyclic, then $K \# J$ is not of order 4.

\bibliographystyle{plain}

\end{document}